\documentclass[10pt, a4paper]{amsart}
\usepackage{amsthm,amssymb, amscd}
\usepackage{mathrsfs, verbatim}
\theoremstyle{plain}
\newtheorem{thm}{Theorem}
\newtheorem{prop}{Proposition}
\newtheorem{cor}[prop]{Corollary}
\newtheorem{lm}[prop]{Lemma}

\theoremstyle{definition}

\newtheorem{rmk}{Remark}
\newtheorem{definition}{Definition}

\theoremstyle{remark}

\newtheorem*{claim*}{Claim}

\newtheorem*{fact}{Fact}

\def\OO{\mathcal{O}}

\def\C{\mathbb{C}}

\def\PP{\mathbb{P}}

\author{A. Alzati}
\address{Alberto Alzati,  Dipartimento di Matematica F.Enriques, Universit\`a di Milano,
via Saldini 50, 20133 Milano (Italy)}
\email{alberto.alzati@unimi.it}

\author{R. Re}
\address{Riccardo Re, Dipartimento di Scienza e Alta Tecnologia, Universit\`a dell' Insubria, Via Valleggio 11,
22100, Como (Italy)}
\email{riccardo.re@uninsubria.it}
\date{18 April, 2018 }
\subjclass[2010]{Primary 13A02, secondary 14N05}
\keywords{Weak Lefschetz property, artinian algebra, complete intersection}
\thanks{This work has been done within the framework of the national project
``Geometry on Algebraic Varieties", Prin (Cofin) 2015 of MIUR}
\title{Complete Intersections of Quadrics and the Weak Lefschetz Property}
\begin{document}
\maketitle
\begin{abstract} We consider graded artinian complete intersection algebras $A=\C[x_0,\ldots,x_m]/I$ with $I$ generated by homogeneous forms of degree $d\geq 2$. We  show that the general multiplication by a linear form $\mu_L:A_{d-1}\to A_d$ is injective. We prove that the Weak Lefschetz Property for holds for any c.i. algebra $A$ as above with $d=2$ and $m\leq 4$. \end{abstract}
\section{Introduction}
The Weak Lefschetz Property (in short WLP) of a graded algebra $A$, presented as the quotient $A=k[x_0,\ldots,x_m]/I$, with $I$ an homogeneous ideal, asserts that for any general linear form
$L=\sum a_ix_i$ and any $k\geq 0$, the multiplication map $\mu_L\colon A_k\to A_{k+1}$ has maximal rank.
It has been conjectured that if $\operatorname{char}(k)=0$ then any complete intersection $k$-algebra, i.e. an algebra as above with $I$ generated by a regular sequence, has the WLP. It is also conjectured that such an algebra should have the Strong Lefschetz Property, i.e. the analogous maximal rank property with $L$ replaced by $L^d$, for any $d$. The most significant case of these conjectures is the case when $A$ is artinian.
These are considered to be very challenging problems, despite affirmative answers for $m\leq 2$, see \cite{Mig2003}, and in the monomial case, see \cite{stanley} and \cite{watanabe}.

In this article we examine the case of a c.i. artinian algebra $A$ with $I$ generated in degree $d\geq 2$. 
The  first main result of the present article is an easy and geometrical proof of a generalization of the so called {\em Injectivity Lemma} of Proposition 4.3 of \cite{MN}. More precisely we show the injectivity of the general multiplication map $\mu_L:A_{d-1}\to A_d$. This result is stated in Corollary \ref{cor:injlemma}. That result is also used in Remark \ref{rmk:d2m3} to obtain a new and simple proof of the WLP for $d=2$ and $m=3$, already covered in \cite{MN}.

Then we examine the case $d=2$ and $m=4$, in which case we prove the WLP by introducing some geometrical methods that seem to be new in the existing literature. This is the content of the other main result of this paper, Theorem \ref{thm:WLPm=4} of Section \ref{sec:d=2m=4}. 

Despite we are proposing only limited progress toward the WLP conjecture, we hope that the methods of the present article can shed some light on the geometrical aspects of the general problem and inspire further investigations.

The plan of the article is the following. In Section \ref{sec:setup} we state the general results on artinian Gorenstein and c.i. algebras that are common knowledge for commutative algebraists and algebraic geometers. 
 Then, in Section \ref{sec:red_I2} we study some interesting stratifications of the projective space $\PP(I_d)$, associated to generating piece of the ideal $I$. These results are immediately used to give a prof of the Injectivity Lemma mentioned above.
 In Section \ref{sec:diff} we introduce our general geometrical method for approaching the WLP conjecture, that is we study the projectivization of variety of pairs $(z,Q)\in A_1\times A_s$ such $zQ=0$. We introduce some differential methods to this purpose. In the final sections \ref{sec:d=2} and \ref{sec:d=2m=4} we consider the case $d=2$ of the WLP conjecture for complete intersections in equal degrees, and give its solution for $m=4$, as already mentioned above. 
\section{General setup}\label{sec:setup}
\subsection*{Notations} Let $f_0,\ldots,f_m$ be homogeneous polynomials in $\C[x_0,\ldots,x_m]$ with no common zeros and of equal degree $d>1$, i.e. a regular sequence. 
 We set $I= (f_0,\ldots,f_m)\subset\C[x_0,\ldots,x_m]$ and we denote $A=\C[x_0,\ldots,x_m]/I$ the artinian quotient ring.  We set $U^\ast=\langle x_0,\ldots,x_m\rangle$, so that $\C[x_0,\ldots,x_m]=S^\bullet U^\ast=\bigoplus_{i=0}^\infty S^iU^\ast$. Observe that $U^\ast\cong A_1$, since we have $d>1$.
 
 Given a $\C$-vector space $V$, we will denote $\PP(V)$ the projective space of the $1$-dimensional vector subspaces of $V$. It $v\in V$ is a non-zero element, we will denote $[v]\in\PP(V)$ its associated point.
\begin{definition}\label{def:WLP} $A$ is said to have the Weak Lefschetz Property (WLP), if for any $k\in \mathbb{N}$ and for any fixed general linear form $L\in U^\ast$ the multiplication map $\mu_L\colon A_{k}\longrightarrow A_{k+1}$ has maximal rank. \end{definition}
Note that the property holds trivially if $k\leq d-2$, since in these cases either $I_{k}=I_{k+1}=0$ and hence $A_k=S^kU^\ast$ and $A_{k+1}=S^{k+1}U^\ast$, or if $k\gg 0$, since then one has $A_{k+1}=0$ . 
 
  \begin{definition} We say that $A$ fails the WLP because of surjectivity in degree $k>0$  if one has $\dim A_{k}\geq \dim A_{k+1}$ but the multiplication $\mu_L\colon A_{k}\to A_{k+1}$ is not surjective for any linear form $L$. In analogous way one defines the failure of WLP because of injectivity. In general we say that $A$ has the WLP in degree $k>0$ if for a general linear form $L$ the multiplication map $\mu_L\colon A_{k}\to A_{k+1}$ has maximal rank. 
  \end{definition}
  We collect a few well known facts in the following proposition.
 \begin{prop}\label{prop:dualWLP} The following facts hold.
 \begin{enumerate}
 \item If $\mu_L:A_{k}\to A_{k+1}$ is injective, then $\mu_L:A_{j}\to A_{j+1}$ is injective for any $j\leq k$.
 \item $A$ is self-dual as it is Gorenstein. As a consequence it is enough to check the WLP in the case of injectivity (resp. in the case of surjectivity).
 \item If $A$ is a c.i. algebra of type $(d,\ldots,d)$ then $\operatorname{socle}(A)=A_{(m+1)(d-1)}\cong \C$ and WLP holds for $A$ if for a general  linear form $L$ the multiplication map $\mu_L:A_{s}\to A_{s+1}$ is injective, with $s=\lceil (m+1)(d-1)/2\rceil -1$.  
 \end{enumerate}
 \end{prop}
 \subsection{The Macaulay inverse system $I^{-1}$}\label{subsec:inverse}
 It is appropriate also to mention at this point that the dual $A^\ast=\oplus A_i^\ast$ is a sub $S^\bullet U^\ast$ module of $S^\bullet U$, with respect the derivation action of $S^\bullet U^\ast$ on $S^\bullet U$. Indeed the action of a linear form $L\in U^\ast$ defined as the traspose of the multiplication map turns out to be a derivation, and indeed the Macaulay-Matlis duality more generally identifies $S^kU^\ast$ with the space of homogeneous linear differential operators  of degree $k$, with constant coefficients, acting on $S^\bullet U$. Then it is also clear that $A^\ast=\operatorname{Ann}(I)$ and it is usually indicated with $I^{-1}$, which is known as the {\em Macaulay inverse system} associated to $I$. The inverse system of any Gorenstein algebra $A$ is generated by a single element $g\in S^MU$ as a $S^\bullet U^\ast$-module, with $M$ equal to the socle degree of $A$. This means that $I^{-1}$ is linearly generated by derivatives of $g$ of any order.  In coordinates, setting $U=\langle u_0,\ldots, u_m\rangle$, with $(u_i)$ the dual  basis of $(x_i)$, then for any 
 $k\geq 0$ one has
 $$(I^{-1})_{M-k}=\langle \partial_{u_0}^{i_0}\cdots\partial_{u_m}^{i_m}(g)\ |\ i_0+\cdots+i_m=k\rangle.$$
 In our case we have $M=(m+1)(d-1)$. A basic, and possibly known, result for $I^{-1}$ that will be useful for us is the following.

 \begin{prop}\label{prop:T} The ideal $I=(f_0,\ldots,f_m)\subset \C[x_0,\ldots,x_m]$, with $f_i$ linearly independent and of degree $d$, is a c.i. ideal if and only if $I^{-1}_d$ does not contain any non-zero power of a linear form $p^d\in S^dU$, and in this case $I^{-1}_d$ has dimension ${m+d\choose d}-m-1$. 
 \end{prop} 

 \begin{proof}  $I$ is c.i. if and only if $V(f_0,\ldots,f_m)=\emptyset$, if and only if there is no $[p]\in \PP(U)$ such that $f_0(p)=\cdots=f_m(p)=0$. But it is well known that $d!f_i(p)=\langle f_i,p^d\rangle$, with $\langle\ ,\rangle\colon S^dU^\ast\otimes S^dU\to \C$ the duality pairing. This shows that $I$ is c.i. if and only if there is no $[p^d]\in\PP(I_d^{-1})$.
Moreover $I_d^{-1}=\operatorname{Ann}(I_d)$ with respect to the duality pairing, and this fact explains the dimension formula. 
 \end{proof}

\begin{rmk}\label{rmk:vertexes} Observe that the non-zero homogeneous polynomials $q\in S^\bullet U$ define hypersurfaces  $V(q)\subset \PP(U^\ast)$.
 To any $L\in U^\ast$ we associate a derivative $\partial_L(q)=L(q)$, defined by the action of $S^\bullet U^\ast $ on $S^\bullet U$, and therefore
 $L(q)=0$ means that the hypersurface $V(q)$ is a cone with a vertex point in $[L]\in\PP(U^\ast)$. 
 \end{rmk}
 \begin{rmk} The inverse system $I^{-1}$ in association to the study of the WLP of the algebra $A=S^\bullet U^\ast/I$ has been already considered in the literature, for example it is a main theme of the article \cite{MOR}. 
\end{rmk}

 \section{Stratifications of $I_d$.}\label{sec:red_I2}
 In this section we collect some easy results on the homogeneous part of the ideal $I$ in degree $d$.
 We will denote $$\begin{array}{l}
 \nu_d(\PP(U^\ast))=\{[z^d]\in S^dU^\ast\ |\ [z]\in\PP(U^\ast)\}\\
 S=\{[Q]\in \PP(S^dU^\ast)\ |\ \exists z\in U^\ast,\exists P\in S^{d-1}U^\ast:\ Q=zP\}\end{array}.$$ More generally, for any $r\geq 1$, we denote 
 $$Sec_{r-1}(S)=\{[z_1P_1+\cdots+z_rP_r]\ |\ z_i\in U^\ast,\ P_i\in S^{d-1}U^\ast\}.$$
 \begin{prop}\label{prop:I_d} If $I=(q_0,\ldots,q_m)$ is a a c.i. ideal generated by forms of degree $d\geq 2$, then  $ \PP(I_d)\cap\nu_d(\PP(U^\ast))$ is a finite set and $\dim \PP(I_d)\cap S\leq 1$. More generally one has $\dim \PP(I_d)\cap Sec_{r-1}(S)\leq \min(2r-1,m)$.
 \end{prop}
 \begin{proof}  Assume $\dim  \PP(I_2)\cap\nu_2(\PP(U^\ast))\geq 1$. Consider a irreducible curve $C$ contained in that set, and consider a local analytic parametrization of this curve at a general point of the form $z_\lambda^d=(z_0+\lambda z_1+O(\lambda^2))^d=z_0^d+d\lambda z_0^{d-1}z_1+O(\lambda^2)$. Then $z_0^d, z_0^{d-1}z_1\in I_d$, impossible because otherwise they could be completed to a basis of $I_d$, so they should be a c.i., but they vanish on $V(x_0)$, a contradiction.
 
  Similarly assume that $\dim \PP(I_d)\cap S\geq 2$. Then there exists a local analytic parametrization of a surface in $\PP(I_d)$ of the form
  $$z_{\lambda,\mu}P_{\lambda,\mu}=(z_0+\lambda z_1+\mu z_2)(P_0+\lambda P_1+\mu P_2)+o(\lambda,\mu),$$ which would give
  $$z_0P_0, z_0P_1+z_1P_0,z_0P_2+z_2P_0 \in I_d.$$
  This produces a contradiction by similar reasons as above, since the three elements above are not a c.i., as they vanish on $V(z_0,P_0)$.
  In the general case, let us assume $m\geq \dim \PP(I_d)\cap Sec_{r-1}(S)\geq 2r$ and let us write a bijective local analytic parametrization  of such a family at a general point  $$\sum_{i=1}^{r}z_i(\lambda)P_i(\lambda)\in I_d,$$ with $\lambda=(\lambda_1,\ldots,\lambda_{2r})$. Then one may expand $$\begin{array}{l} z_i(\lambda)=z_i+\sum_j\lambda_j z_{ij}+o(\lambda)\\ P_i(\lambda)=P_i+\sum_j\lambda_j P_{ij}+o(\lambda)\end{array}$$ and find the relations $$\begin{array}{l} \sum_i z_iP_i\in I_d \\\
  \sum_i(z_iP_{ij}+z_{ij}P_i)\in I_d,\quad \forall\ j=1,\ldots,2r.\end{array}$$
Now observe that the $2r+1$ forms above are independent by construction and they all vanish on $V(z_1,\ldots,z_r,P_1,\ldots,P_r)$, which has codimension $\leq 2r$ in $\PP^m$, a contradiction, since they should be a complete intersection. \end{proof}
An immediate consequence of the proposition above significantly extends the {\em Injectivity Lemma} of Proposition 4.3 of \cite{MN}, proved in that paper in the case $d=2$.
\begin{cor}[Injectivity Lemma]\label{cor:injlemma} Let $A=\C[x_0,\ldots,x_m]/I$ an artinian algebra, with $I$ generated by a regular sequence of forms of degree $d$. Then for general $z\in A_1$ the multiplication map $\mu_z:A_{d-1}\longrightarrow A_d$ is injective. 
\end{cor}
\begin{proof} Otherwise for any general $[z]\in\PP(U^\ast)$ there exists some $P\in S^{d-1}U^\ast$ such that $zP\in I_d$, (note that such factorization is unique) hence $\dim(\PP(I_d)\cap S)\geq m=\dim\PP(U^\ast)$, but this contradicts the second dimension statement of Proposition \ref{prop:I_d}. \end{proof}

\begin{rmk}\label{rmk:d2m3} In \cite{MN} it was observed that, as a consequence of the result above in the case $d=2$, the WLP holds for $d=2$ and $2\leq m\leq 3$. Indeed in those cases one has $s=\lceil(m+1)(d-1)\rceil-1=1$.
\end{rmk}

\begin{rmk} Proposition \ref{sec:red_I2} actually gives more precise information than Corollary \ref{cor:injlemma} on the dimension so-called {\em non Lefschetz locus}, that is the locus of $[z]\in \PP(A_1)$ such that the moltiplication map $\mu_z:A_{d-1}\to A_d$ is not injective, by showing that this locus has dimension at most $1$. We refer the reader to the recent article \cite{BoMiRoNa} for many other deep results in this context.
\end{rmk}

 In the case of $d=2$, since quadrics are classified by their rank up to projective transformations, we can give a more precise and alternative form
 of Proposition \ref{prop:I_d}, which will be useful later to cover the case $d=2$ and $m=4$.
 \begin{prop}\label{prop:stratI2} Set $R_r=\{ [Q]\in \PP(S^2U^\ast)\ |\ \operatorname{rk}(Q)\leq r\}$ for any $1\leq r\leq m$, that is, $R_r$ is the projectivized set of quadrics of rank at most $r$. Let $I$ a c.i. ideal generated by $m+1$ quadrics. Then $\dim \PP(I_2)\cap R_r\leq r-1$.
 \end{prop}
 \begin{proof} We can use induction on $r$. The case $r=1$ is covered already in Proposition \ref{sec:red_I2}. Assume that $\dim\PP(I_2)\cap R_r\geq r>1$ and let $Q$ be a general point in a $r$ dimensional irreducible subvariety of that closed set. Hence, by the inductive hypothesis, it must be $\operatorname{rk}(Q)=r$ and therefore we can write $Q=L_1^2+\cdots+L_r^2$. Since the set of quadratic forms of rank $r$ is a single orbit under the action of $GL(m+1)$, we can find a local analytical parametrization of the given $r$-dimensional component of $\dim\PP(I_2)\cap R_r$ with parameters $\epsilon=(\epsilon_1,\ldots,\epsilon_r)$ such that 
 $$Q_\epsilon=\sum_{i=1}^r(L_i+\epsilon_1L_{i1}+\cdots+\epsilon_rL_{ir})^2\in \PP(I_2)\cap R_r.$$
 Moreover, by construction, we can assume that $Q$, $F_1=\partial_{\epsilon_1}Q_\epsilon|_{\epsilon=0},\ldots,F_r=\partial_{\epsilon_r}Q_\epsilon|_{\epsilon=0}$ are independent. Then, computing the derivatives above, we find that $$Q=\sum L_i^2,\ F_1=\sum L_iL_{i1},\ \ldots,\ F_r=\sum_iL_iL_{i,r}$$ are $r+1$ independent elements in $\PP(I_2)$ that vanish on $V(L_1,\ldots, L_r)\not=\emptyset$. Since $r\leq m$ and any base of $I_2$ must be a c.i., we obtain a contradiction.
 \end{proof}
\begin{rmk}
The results of this section, although very simple, strongly depend on the complete intersection hypothesis. They cannot in general be extended to the case of Gorenstein algebras $A$, even when these algebras are presented by quadrics. Indeed infinite series of counterexamples to the result in Corollary \ref{cor:injlemma} in the Gorenstein case have been provided in \cite{GoZa}, see for example their  Corollary 3.8.
\end{rmk}
 \section{Differential lemmas}\label{sec:diff}
 Set $s=\lceil(m+1)(d-1)/2\rceil-1$ as above. We introduce the variety
 $$\Gamma=\{([z],[Q])\in \PP(A_1)\times\PP(A_s)\ |\ zQ=0\},$$ which inherits from $\PP(A_1)\times\PP(A_s)$ the projections $\pi_1$ and $\pi_2$ to $\PP(A_1)$ and $\PP(A_s)$, respectively. 
 Assuming that WLP does not hold for the algebra $A$, that is the general multiplication map
 $\mu_z\colon A_s\to A_{s+1}$ is not injective, there exists a subvariety  $$\mathcal{W}\subset \Gamma,$$
defined as the unique component of $\Gamma$ with $\pi_1(\mathcal{W})=\PP(A_1)$. The uniqueness of $\mathcal{W}$ comes from the fact that the fibers of the projections of the variety $\Gamma$ are linear spaces.

 For general $[z]\in \PP(A_1)$ and $[Q]\in \pi_2(\mathcal{W})$, we define the vector spaces
 $$\begin{array}{c}\mathcal{Q}(z)=\{Q\in A_s\ |\ zQ=0\}\\
 \mathcal{Z}(Q)=\{z\in A_1\ |\ zQ=0\},\end{array}.$$
 This means that  $\pi_1^{-1}([z])=\PP\mathcal{Q}(z)$ and $\pi_2^{-1}([Q])=\PP\mathcal{Z}(Q)$.
We set the following notations, for general $[z]\in \PP(A_1)$ and general $[Q]\in\PP(A_1)$ such that $zQ=0$.

\begin{equation}\label{eq:deltaepsilon}
\begin{array}{c} \delta=\dim\pi_2^{-1}([Q])=\dim\mathcal{Z}(Q)-1,\\ \varepsilon=\dim\pi_1^{-1}([z])=\dim\mathcal{Q}(z)-1 .
\end{array}
\end{equation}
We also set \begin{equation}\label{eq:dimW}
N=\dim\mathcal{W}=m+\varepsilon=\dim\pi_2(\mathcal{W})+\delta.\end{equation}

Now consider a general point $([\bar{z}],[\bar{Q}])\in\mathcal{W}$ and a system of parameters $\lambda=(\lambda_1,\ldots,\lambda_N)$ for $\mathcal{W}$ centered at $([\bar{z}],[\bar{Q}])$. We denote $([z],[Q])$ a point in a neighborhood of $([\bar{z}],[\bar{Q}])$, with the understanding that $z=z(\lambda_1,\ldots,\lambda_N)$ and $Q=Q(\lambda_1,\ldots,\lambda_N)$ are analytic functions defined at $\lambda=0$. Finally, we denote
$$z_i=\partial_{\lambda_i}z,\quad Q_i=\partial_{\lambda_i}Q, \quad i=1,\dots,N.$$

\noindent
\begin{lm}\label{lm:diff} Under the notations above, one has $\dim\langle z_1,\ldots,z_N\rangle=N-\varepsilon=m$ and $\dim\langle Q_1,\ldots,Q_N\rangle=N-\delta$. In particular one has $\langle z, z_1,\ldots,z_N\rangle=U^\ast$.
Moreover, the following relations hold, for any $i,j=0,\ldots,N$.
\begin{enumerate}
\item[\it{i})] If $Q,Q'\in\pi_1^{-1}([z])$, then $QQ'=0\in A_{2s}$,
\item[\it{ii})] $Q\left(\frac{\partial Q}{\partial \lambda_i}\right)=0\in A_{2s}$,
 \item[\it{iii})] $z\left(\frac{\partial Q}{\partial \lambda_i}\right)\left(\frac{\partial Q}{\partial \lambda_j}\right)=0\in A_{2s+1}$,
  \item[\it{iv})] For any $h\geq 0$ and for any $1\leq i_1,\ldots,i_h\leq N$ one has $$z^{h+1}\left(\frac{\partial^h Q}{\partial \lambda_{i_1}\cdots\partial \lambda_{i_h}}\right)=0\in A_{s+h}.$$ In particular, for any $i$, one has $z^2\frac{\partial Q}{\partial \lambda_i}=0$. \end{enumerate}
\end{lm}
\begin{proof}  The first assertions of the Lemma are clear by the fact that $\pi_1(\mathcal{W})=\PP(U^\ast)$, and by the surjectivity of the  tangent maps for $\pi_1$ and $\pi_2$ at a general point $([z],[Q])\in\mathcal{W}$.

By applying $\frac{\partial}{\partial\lambda_i}$ to the relation $zQ=\sum c_j(\lambda)f_i(x)\in I_{s+1}$, with $x=(x_0,\ldots,x_m)$ a fixed base of $U^\ast$ and $f_j(x)$ a generating set for $I_{s+1}\subset S^{s+1}U$,
we immediately see that
\begin{equation}\label{eq:basicrel2}
z_iQ+z\frac{\partial Q}{\partial\lambda_i}\in I_{s+1}.
\end{equation}
Multiplying the relation above by $Q'$ and using $zQ'\in I$, we obtain
$z_iQQ'\in I_{2s+1}$ and, since $A_1=\langle z,z_1,\ldots,z_m\rangle$, we see that $A_1\cdot QQ'=\{0\}\subset A_{2r}$. By our choice of $s$ we have $\operatorname{socle}(A)_{2s}=\{0\}$, therefore we find $QQ'\in I_{2s}$. This proves (i).

Setting $Q=Q'$, by derivations of the last relation, we get $Q\frac{\partial Q}{\partial\lambda_j}\in I_{2s}$.
Next, we compute
$z\frac{\partial Q}{\partial\lambda_i}\frac{\partial Q}{\partial\lambda_j}=-z_iQ\frac{\partial Q}{\partial\lambda_j}=0$, using (\ref{eq:basicrel2}).
This completes the proof of {\it i}), {\it ii}) and {\it iii}).

We prove {\it iv}) by induction on $h$, (we assume that it is true for $h$ and we prove it for $h+1$) the base case being the starting relation $zQ=0$. Let us assume by inductive hypothesis $$z^{h+1}F\in I_{s+h},\quad F=\frac{\partial^h Q}{\partial \lambda_{i_1}\cdots\partial \lambda_{i_h}}.$$ By derivation with respect to $\lambda_i$ we obtain 
$$(h+1)z^hz_i F+z^{h+1}\frac{\partial F}{\partial\lambda_i}\in I_{s+h}.$$ Then the inductive step is immediately proved by multiplying this last relation by $z$ and using that $z^{h+1}F\in I_{s+h}$.

\end{proof}
 \section{complete intersection algebras presented by quadrics}\label{sec:d=2}
 We recall the following simple formula for Hilbert function of an artinian c.i. algebra presented by quadrics. 
 \begin{fact}\label{fact:HF} Let $A$ be the artinian algebra obtained from a complete intersection of quadrics in $\PP^m$. Then one has $HF(A,k)=\dim A_k={m+1\choose k}$ for any $k\geq 0$.
 \end{fact}
Indeed one can compute $\dim A_k$ from the Koszul resolution
 $$0\to S(-2(m+1))\otimes \bigwedge^{m+1}U^\ast\to \cdots\to S(-2)\otimes \bigwedge^1U^\ast\to S\to A\to 0,$$ and
  the fact that this resolution gives as result $\dim A_k={m+1\choose k}$ can be seen directly considering the special case $I=(x_0^2,\ldots,x_m^2)$, in which case $A_k$ has basis given by the classes mod $I$ of all squarefree monomials $x_0^{\epsilon_0}\cdots x_m^{\epsilon_m}$, i.e. with exponents  $\epsilon_i\in\{0,1\}$, whose number is exactly ${m+1\choose k}$. 
\vskip2mm
 The following technical lemma will be useful later.

 \begin{lm}\label{lm:bdeltageq1} Let $A$ be a c.i. artinian algebra as above, with $m\geq 4$. Then the following properties hold.
\begin{enumerate}
\item For any two distinct points $[z],[w]\in \PP(U)$ one has $\dim \langle z,w\rangle A_1\geq 2m-1$.
\item Denoting $\overline{\varUpsilon}$ the subvariety of $Gr(m-2,\PP(U))$ whose points are the $(m-2)$-planes $\Pi=V(z,w)\in Gr(m-2,\PP(U))$ such that $\dim \langle z,w\rangle A_1= 2m-1$, then $\dim \overline{\varUpsilon}\leq 3$. 
\end{enumerate}
\end{lm}
\begin{proof} {(1).} One has 
$\langle z,w\rangle A_1=(\langle z,w\rangle U^\ast+I_2)/I_2\cong \langle z,w\rangle U^\ast/(\langle z,w\rangle U^\ast\cap I_2)$.
Note that $\dim\langle z,w\rangle U^\ast=2(m+1)-1=2m+1$, as it is the degree $2$ piece of the ideal generated by $z,w$. Therefore, the general bound  $\dim \langle z,w\rangle A_1\geq 2m-1$ is equivalent to $\dim(\langle z,w\rangle U^\ast\cap I_2)\leq 2$. But this latter fact is clearly true, because $I_2$ is generated by a complete intersection and therefore at most $2$  independent element of $I_2$ can define hypersurfaces that contain the $2$-codimensional linear space $\Pi=V(z,w)$. 
\paragraph{(2)}
By the calculations in the proof of statement (1), the only possibility for obtaining $\dim \langle z,w\rangle A_1= 2m-1$ is $$\dim\PP(\langle z,w\rangle U^\ast\cap I_2)=1.$$ Let us introduce the incidence variety
$$\mathcal{Y}=\{(\Pi,[F])\ |\  \Pi\in \operatorname{Gr}(m-2,\PP(U)),\ [F]\in \PP(I_2),\ \Pi \subset V(F)\},$$ endowed with the two projections 
$$\pi_1(\Pi,[F])=\Pi,\quad \pi_2(\Pi,[F])=[F]$$ to $Gr(m-2,\PP(U))$ and $\PP(I_2)$, respectively. Let also introduce the subvariety
$$\varUpsilon=\{(\Pi,[F])\in\mathcal{Y}\ |\ \dim\pi_1^{-1}(\Pi)=1\}.$$
Then the variety in the statement is $$\overline{\varUpsilon}=\pi_1(\varUpsilon).$$
Note that, setting $\Pi=V(z,w)$, then $\PP(\langle z,w\rangle U^\ast\cap I_2)=\pi_2\pi_1^{-1}(\Pi)$. As we have shown above, these spaces are all lines for any $\Pi\in\overline{\varUpsilon}$ and hence the irreducible components of $\varUpsilon$ are the inverse images of the irreducible components of $\overline{\varUpsilon}=\pi_1(\varUpsilon)$. We denote
$$\overline{\varUpsilon}=\overline{\varUpsilon}_1\cup\cdots\cup\overline{\varUpsilon}_r, \quad \varUpsilon=\varUpsilon_1\cup\cdots\cup\varUpsilon_r,\quad \varUpsilon_i=\pi_1^{-1}(\overline{\varUpsilon}_i)$$ the corresponding decomposition into irreducible components. Observe that $\dim\varUpsilon_i=\dim\overline{\varUpsilon_i}+1$.

 Note also that  $(\Pi,[F])\in\varUpsilon$ implies $F=zL+wM$, with $\Pi=V(z,w)$, which implies that $[F]\in R_4$, and of course $\Pi$ is one of the $(m-2)$-planes contained in $V(F)$. Hence $\pi_2(\varUpsilon)\subseteq \PP(I_2)\cap R_4$ and for any $[F]\in \pi_2(\varUpsilon)$ one has $$\pi_1(\pi_2|_{\varUpsilon})^{-1}([F])\subseteq \{\Pi\ |\ \Pi\subset V(F)\}.$$

 If for some $\Pi\in\overline{\varUpsilon}$ it were $\pi_2\pi_1^{-1}(\Pi) \subseteq R_2$, then $\dim\PP(\langle z,w\rangle U^\ast\cap I_2)=\dim\pi_2\pi_1^{-1}(\Pi)=1$ means that $\Pi=V(z,w)$ is  contained in a pencil of rank $2$ (and hence reducible) quadrics  $F_{\lambda,\mu}\in I_2$. Such a pencil has the form either $(\lambda z+\mu w)L$ or $(\alpha z+\beta w)M_{\lambda,\mu}$. In any case the given pencil has a fixed hyperplane component, and therefore it cannot be contained in $I_2$, since this latter is generated by a complete intersection. So $\pi_2\pi_1^{-1}(\Pi) \subseteq R_2$ is impossible. Then we have that
 for any component $\varUpsilon_i$ of $\varUpsilon$ a general point $(\Pi,[F])\in\varUpsilon_i$ is such that $\operatorname{rk}(F)\geq 3$.
 \vskip2mm
 
 It is well known that any quadric $[F]$ of rank  $3$ or $4$ has the family of linear spaces $\Pi\cong\PP^{m-2}$ contained in $V(F)$ of  dimension $1$, hence for a general $[F]\in\pi_2(\varUpsilon_i)$ one has $(\dim\pi_2|_\varUpsilon)^{-1}([F])\leq 1$. Therefore we have shown that for any $i=1,\ldots,r$ one has
 \begin{equation}\label{eq:dimUpsiloni}\dim\varUpsilon_i=\dim(\overline\varUpsilon_i)+1\leq \dim\pi_2(\varUpsilon_i)+1.\end{equation}
By Proposition \ref{prop:stratI2} we know that  $\dim \PP(I_2)\cap R_4\leq 3$ and, since $\pi_2(\varUpsilon_i)\subseteq \PP(I_2)\cap R_4$, we have  $\dim\pi_2(\varUpsilon_i)\leq 3$ and therefore $\dim\varUpsilon_i\leq 4$ and $\dim\overline{\varUpsilon}_i\leq 3$. 
 
\end{proof}
 \section{Application: the case $d=2$ and $m=4$}\label{sec:d=2m=4}
 In this last section we apply the results obtained so far to prove that the WLP holds for c.i. algebras presented by quadrics in $\PP^4$, that is we assume $d=2$ and $m=4$. Even this simple case appears not to be covered in the existing literature. For $d=2$ and $m=4$ the dimensions of $A_i$ for $i=0,\ldots,5$ are $$1,5,10,10,5,1$$
By the results stated in Proposition \ref{prop:dualWLP}, we only need to examine the general multiplication map $$\mu_z\colon A_2\to A_3,$$ as in the present case one has $s=\lceil(m+1)(d-1)/2\rceil-1=2$.
 
 Note that by letting $[z]$ vary in $\PP(A_1)$, we can define a sheaf morphism
 \begin{equation}\label{eq:Kbundle} A_2\otimes\OO_{\PP(A_1)}\stackrel{\Phi}\longrightarrow A_3\otimes \OO_{\PP(A_1)}(1),
 \end{equation}
 with $\Phi$ defined fiberwise by $\Phi([z])=\mu_z$.
 Note that $\Phi$ symmetric, that is, identifying $A_3\cong A_2^\ast$ by means of the multiplication pairing $A_3\otimes A_2\to A_5\cong\C$, then one can write $\Phi$ as a sheaf map $\Phi\colon A_2\otimes\OO_{\PP^1}\to A_2^\ast\otimes\OO_{\PP^1}(1)$, and then $\Phi=\Phi^\ast(1)$ .  In particular $\Phi$ can be described a symmetric matrix of linear forms, with respect to a choice of a basis for $A_2$ and its dual basis for $A_3\cong A_2^\ast$.
 
Assuming that the WLP does not hold, let $([z],[Q])\in\mathcal{W}$ be a general element. We consider the multiplication map
 $$\mu_Q:A_2\longrightarrow A_4\cong \C^5.$$
 Under the notations of Section \ref{sec:diff}, we have the following result.
 \begin{lm}\label{lm:cokermuQ} $\dim\operatorname{coker}(\mu_Q)=\dim \mathcal{Z}(Q)=\delta+1$.
 \end{lm}
 \begin{proof} The map $A_2\stackrel{Q}\longrightarrow A_4$ is dual to the map $A_1\stackrel{Q}\longrightarrow A_3$ by the perfect pairings $A_4\otimes A_1\to A_5$ and $A_2\otimes A_3\to A_5$ defined by the multiplication. Hence $\operatorname{coker}(\mu_Q)\cong \mathcal{Z}(Q)^\ast$, which proves the statement.
 \end{proof}

 We also have the following formula relating $\mathcal{Z}(Q)$ with the vector space spanned by the derivatives of $Q$ with respect to the parameters $\lambda_1,\ldots,\lambda_N$, with $N=\dim\mathcal{W}$, introduced in the previous section.

 \begin{lm}\label{lm:spanQi}  For arbitrary $m$ one has $$\dim \langle Q,Q_1,\ldots,Q_N\rangle= m+1+\varepsilon-\delta.$$ In particular for $m=4$ we have $$\dim\langle Q,Q_1,\ldots,Q_N\rangle = 5+\varepsilon-\delta.$$
 \end{lm}
 \begin{proof}
 The embedded tangent space to the image of the map $\pi_2:\mathcal{W}\longrightarrow \PP(A_2)$ defined by $([z],[Q])\mapsto Q\in \PP(A_2)$ is given by $\langle Q, Q_1,\ldots,Q_N\rangle/\langle Q\rangle$, hence it has dimension equal to $\dim\langle Q,Q_1,\ldots,Q_N\rangle-1$.  By Lemma \ref{lm:diff} the same tangent space has dimension $N-\delta=m+\varepsilon-\delta$, from which the statement follows.\end{proof}

 We introduce one last preliminary result about the dimension of $\mathcal{Z}(Q)$.  This will be the only place in this paper where it turns out very useful to consider the inverse system $I^{-1}$, mentioned in Section \ref{sec:setup}.

  \begin{lm}\label{lm:ZQbound}
 If $m=4$  one has $\delta\leq 1$, that is $\dim \mathcal{Z}(Q)\leq 2$. \end{lm}
\begin{proof} Assume that $\dim \mathcal{Z}(Q)\geq 3$. Then consider $g\in S^5U$ such that $I^{-1}$ is generated by $g$ as a $S^\bullet U^\ast$-module.  Since $zQ=0$ for any $z\in\mathcal{Z}(Q)$, then, by Remark \ref{rmk:vertexes}, the cubic $Q(g)\in S^3U$ is a cone with vertex space $\PP\mathcal{Z}(Q)$, and we are assuming $\dim\PP\mathcal{Z}(Q)\geq 2$. 
But then, in a suitable coordinate system, $Q(g)$ is defined by a degree three homogeneous polynomial in two variables, which therefore may written in a suitable coordinate system as $f(x,y)=x(\alpha_1x+\beta_1y)(\alpha_2x+\beta_2y)$. Then one can easily see the vector space generated by $f_x$ and $f_y$, contained in $I^{-1}_2$ by the construction resumed  in section 2.1, always contains the square of some linear form, obtaining a contradiction by Proposition \ref{prop:T}.
\end{proof}
 The spaces $\mathcal{Z}(Q)$ and $\langle Q, Q_1,\ldots,Q_N\rangle$ are connected with $\mu_Q$ in the following way.
 \begin{lm}\label{lm:kermuQ} One has $\mathcal{Z}(Q)A_1+\langle Q, Q_1,\ldots,Q_N\rangle\subseteq \ker\mu_Q$.
 \end{lm}
 \begin{proof} This is clear by definition of $\mathcal{Z}(Q)$ and by Lemma \ref{lm:diff} (i) and (ii).
 \end{proof}
We can finally prove the following.
\begin{thm}\label{thm:WLPm=4} The WLP holds for $m=4$ and $d=2$.
\end{thm} 
\begin{proof}
Assume that WLP does not hold. In view of Lemma \ref{lm:ZQbound}, we have two cases.
\vskip2mm
\paragraph{\bf Case 1: $\delta= 1$} Then, by Lemma \ref{lm:cokermuQ} we have $$\dim\ker\mu_Q=\dim A_2-\dim A_4+\dim\operatorname{coker}\mu_Q=5+\delta+1=7.$$ Since $\mathcal{Z}(Q)A_1\subseteq \ker\mu_Q$ and $\dim \mathcal{Z}(Q)A_1\geq 7$ by Lemma \ref{lm:bdeltageq1}, we have $$\dim \mathcal{Z}(Q)A_1=7.$$ 
Then Lemma \ref{lm:bdeltageq1} applies to this case.  The planes $\Pi=V(z,w)$, with $\langle z,w\rangle=\PP(\mathcal{Z}(Q))=\pi_1^{-1}([Q])$, for $Q$ varying in $\pi_2(\mathcal{W})$, form a irreducible family of dimension  equal to $\dim\pi_2(\mathcal{W})=4+\varepsilon-\delta=3+\varepsilon$, by formula (\ref{eq:dimW}) in Section \ref{sec:diff}. But by (2) of Lemma \ref{lm:bdeltageq1}, we know that such a family is actually $3$-dimensional, hence $\varepsilon=0$.
In this case we have  $\dim\pi_2(\mathcal{W})=3$ and $\dim\mathcal{W}=4$. A general fiber of $\pi_1:\mathcal{W}\to \PP(A_1)=\PP^4$ is one point and for any $[Q]\in\pi_2(\mathcal{W})$ one has $\pi_1\pi_2^{-1}([Q])$ a line in $\PP(A_1)$. Moreover $\PP(A_1)$ is clearly covered with such lines. 

Then we take $[z]\in \PP(A_1)$ general and consider the unique $[Q]$ with $([z],[Q])\in \mathcal{W}$. Of course the line $l=\pi_1\pi_2^{-1}([Q])$ contains  $z$. 

Now let us consider the sheaf map $\Phi$ (\ref{eq:Kbundle}) described at the beginning of Section \ref{sec:d=2m=4}, and let us restrict it to $l\cong \PP^1$. 
We obtain an exact sequence of sheaves
$$0\to\mathcal{K}_l\to A_2\otimes\OO_{\PP^1}\stackrel{\Phi_l}\longrightarrow A_3\otimes\OO_{\PP^1}(1)\longrightarrow \mathcal{N}_l\to 0,$$
with $\mathcal{K}_l$ and $\mathcal{N}_l$ defined as the sheaf kernel and cokernel of $\Phi_l$, respectively. 

Now, the fact that $\varepsilon=0$ implies that $\mathcal{K}_l$ is a rank $1$ torsion free sheaf on $\PP^1$, hence  it is a line bundle $\mathcal{K}_l\cong \OO_{\PP^1}(t)$ with $t\leq 0$. A priori the sheaf inclusion $\mathcal{K}_l\hookrightarrow A_2\otimes\OO_{\PP^1}$ may not be a vector bundle embedding, i.e. the associated fiber map $\mathcal{K}_l([w])\to A_2$ may be $0$ at some point $[w]\in \PP^1$. However, since the zero constant $Q\in A_2$ belongs to $\ker\Phi_l([w])$ for any $[w]\in l=\pi_1\pi_2^{-1}([Q])$, we see that the vector bundle embedding $\OO_{\PP^1}\stackrel{Q}\longrightarrow A_2\otimes \OO_{\PP^1}$ factors through $\mathcal{K}_l$, which induces an isomorphism $\OO_{\PP^1}\cong \mathcal{K}_l$.
Then the exact sequence of sheaves above becomes a  vector bundle exact sequence 
$$0\to\OO_{\PP^1}\to A_2\otimes\OO_{\PP^1}\stackrel{\Phi}\longrightarrow A_3\otimes\OO_{\PP^1}(1)\longrightarrow \mathcal{N}_l\to 0.$$
The symmetry of $\Phi$ implies that $\mathcal{N}_l=\mathcal{K}_l^\ast(1)=\OO_{\PP^1}(1)$ and this is impossible by degree reasons, as $\dim A_2=\dim A_3=10$.
\vskip2mm
\paragraph{\bf Case 2: $\delta=0$} Then $\mathcal{Z}(Q)=\langle z\rangle$ and, by Lemma \ref{lm:kermuQ}, we have 
$$zA_1+\langle Q, Q_1,\ldots,Q_N\rangle\subseteq \ker\mu_Q.$$
 Since $z$ is general and hence we can apply the result of Corollary \ref{cor:injlemma}, we have $$\dim zA_1=5.$$ Moreover, by  Lemma \ref{lm:cokermuQ}, we have $$\dim\ker\mu_Q=\dim A_2-\dim A_4+\delta+1=6$$ and finally, by Lemma  \ref{lm:spanQi}, we have $$\dim \langle Q, Q_1,\ldots,Q_N\rangle=5+\varepsilon.$$
 Then it is easy to conclude that $$\dim  (zA_1\cap\langle Q ,Q_1,\ldots,Q_N\rangle)\geq 4+\varepsilon\geq 4.$$
 Now we consider a subspace $V\subseteq \mu_z^{-1}(zA_1\cap\langle Q ,Q_1,\ldots,Q_N\rangle)\subset A_1$ with $\dim V=4$. 
\vskip2mm
 \paragraph{\em Claim} $A_2=VA_1$.
 \begin{proof}[Proof of the Claim] Recalling that $A_1=U^\ast/I_1=U^\ast$, the claim is equivalent to assert that $VU^\ast+I_2=S^2U^\ast$. First of all, using the fact that $V$ is the space of linear forms vanishing on one point $p\in\PP^m=\PP(U)$, that is $V=H^0\mathcal{I}_p(1)$, we see that $VU^\ast=H^0\mathcal{I}_p(2)$ and it has dimension $\dim S^2U^\ast-1=14$. Then $VU^\ast+I_2\subsetneq S^2U^\ast$ if and only if $$\dim (VU^\ast\cap I_2)\geq 1+ \dim I_2+\dim VU^\ast-\dim S^2U^\ast=5.$$  But this means that in $I_2\subseteq VU^\ast=H^0\mathcal{I}_p(2)$, which is impossible. 
    \end{proof}
Then we have $zA_2=zVA_1 \subset \langle Q, Q_1,\ldots,Q_N\rangle A_1$, which, by Lemma \ref{lm:diff}  (iv) implies $$z^3A_2 \subset z^2\langle Q, Q_1,\ldots,Q_N\rangle A_1=\{0\},$$ hence, since the socle of $A$ is generated in degree $5$, one finds $z^3=0$. But this is impossible for general $z\in A_1$, since the $z^3$'s with $z$ general generate $A_3$.
 \end{proof}

\end{document}